\documentclass[reqno]{amsart}
\usepackage{amsmath,amsthm,amssymb,amscd}

\newtheorem{theorem}{Theorem}[section]
\newtheorem{lemma}[theorem]{Lemma}
\newtheorem{prop}[theorem]{Proposition}
\newtheorem{corollary}[theorem]{Corollary}
\theoremstyle{definition}
\newtheorem{definition}[theorem]{Definition}

\theoremstyle{remark}
\newtheorem{remark}[theorem]{Remark}
\newtheorem{conj}[theorem]{Conjecture}
\numberwithin{equation}{section}
\allowdisplaybreaks
\begin{document}

%
%
%
%
%
%
%
%
%

\title[$\text{PDO}_t(n)$ modulo powers of $2$ and $3$]
 {Congruences for the partition function $\text{PDO}_t(n)$ modulo powers of $2$ and $3$}

\author{Gurinder Singh}
\address{Department of Mathematics, Indian Institute of Technology Guwahati, Assam, India, PIN- 781039}
\email{gurinder.singh@iitg.ac.in}

\author{Rupam Barman}
\address{Department of Mathematics, Indian Institute of Technology Guwahati, Assam, India, PIN- 781039}
\email{rupam@iitg.ac.in}

\date{July 10, 2023}



\dedicatory{}

\begin{abstract}  Lin introduced the partition function $\text{PDO}_t(n)$, which counts the total number of tagged parts over all the partitions of $n$ with designated summands in which all parts are odd. For $k\geq0$, Lin conjectured congruences for $\text{PDO}_t(8\cdot3^kn)$ and $\text{PDO}_t(12\cdot3^kn)$ modulo $3^{k+2}$. In this article, we develop a new approach to study these congruences. We study the generating functions of $\text{PDO}_t(8\cdot3^kn)$ and $\text{PDO}_t(12\cdot3^kn)$ modulo $3^{k+3}$ for certain values of $k$. We also study $\text{PDO}_t(n)$ modulo powers of $2$. We establish infinitely many congruences for $\text{PDO}_t(n)$ modulo $8$ and $32$. We prove several congruences modulo small powers of $2$ and discuss the existence of congruences modulo arbitrary powers of $2$ similar to those in Lin's conjecture. In reference to this, we also pose some problems for future work.
\end{abstract}
\maketitle
\section{Introduction and statement of results} 
A partition of a positive integer $n$ is a non-increasing sequence of positive integers, called parts, whose sum is $n$. In \cite{andrews2002}, Andrews, Lewis, and Lovejoy introduced partitions with designated summands. A partition with designated summands of $n$ is obtained from an ordinary partition by tagging exactly one of each part size. The number of such partitions of $n$ is denoted by $\text{PD}(n)$. For example, $\text{PD}(4)=10$ with the relevant partitions being $4',\ 3'+1',\ 2'+2,\ 2+2',\ 2'+1'+1,\ 2'+1+1',\ 1'+1+1+1,\ 1+1'+1+1,\ 1+1+1'+1,\ 1+1+1+1'$. They also studied another partition function $\text{PDO}(n)$ which counts the number of partitions of $n$ with designated summands in which all parts are odd. From the above example, $\text{PDO}(4) = 5$. Later, many
authors have studied these two partition functions (see for example \cite{Baruah2015, chen_ji_2013, xia_2016}). Recently, Lin \cite{Lin2018} introduced two new partition functions $\text{PD}_t(n)$ and $\text{PDO}_t(n)$ related to partitions with designated summands. The partition function $\text{PD}_t(n)$ counts the total number of tagged parts over all the partitions of $n$ with designated summands. For instance, $\text{PD}_t(4) = 13$. The other partition function $\text{PDO}_t(n)$ counts the total
number of tagged parts over all the partitions of $n$ with designated summands in which all parts are odd. For example, $\text{PDO}_t(4) = 6$. Lin found the generating functions of $\text{PD}_t(n)$ and $\text{PDO}_t(n)$. The generating function of $\text{PDO}_t(n)$ is given by \cite[eq. (1.5)]{Lin2018}:
\begin{align}\label{gen_fn_PDOt}
\sum_{n=0}^{\infty}\text{PDO}_t(n)q^n=q\frac{f_2f_3^2f_{12}^2}{f_1^2f_6},
\end{align}
where, for any positive integer $m$, $f_m:=\prod_{j=1}^{\infty}(1-q^{jm})$.
\par Lin \cite{Lin2018} established many congruences modulo small powers of 3 satisfied by $\text{PD}_t(n)$ and $\text{PDO}_t(n)$. For example, he proved the following Ramanujan-type congruences modulo $9$ and $27$ satisfied by $\text{PDO}_t(n)$: For $n\geq0$,
\begin{align}
\label{conj_Lin_0}\text{PDO}_t(8n)&\equiv \text{PDO}_t(12n)\equiv \text{PDO}_t(12n+8)\equiv 0 \pmod{9},\\
\label{conj_Lin_1} \text{PDO}_t(24n)&\equiv \text{PDO}_t(36n)\equiv \text{PDO}_t(36n+24)\equiv 0 \pmod{27}.	
\end{align}
Lin also derived infinitely many congruences for $\text{PDO}_t(n)$ modulo $9$ and $27$. If $p\geq5$ is a prime with $\left(\frac{-3}{p}\right)=-1$, for $n\geq0$ and $1\leq k\leq p-1$, he proved that 
\begin{align*}
\text{PDO}_t(24p^2n+24kp+4p^2)&\equiv0\pmod{9},\\
\text{PDO}_t(72p^2n+72kp+12p^2)&\equiv0\pmod{27}.
\end{align*} 
He further conjectured the following congruences.
\begin{conj}\cite[Conjecture 6.1]{Lin2018}\label{conj_Lin}
For $k,n\geq 0$,
\begin{align}\label{lin_conj1}
\text{PDO}_t(8\cdot 3^k n)&\equiv 0 \pmod{3^{k+2}},\\
\text{PDO}_t(12\cdot 3^k n)&\equiv 0\pmod{3^{k+2}}.\label{lin_conj2}
\end{align}
\end{conj}
The congruences \eqref{conj_Lin_0} and \eqref{conj_Lin_1} establish Conjecture \ref{conj_Lin} for $k=0,1$, respectively. We, along with Ajit Singh \cite{Barman2023}, further proved this conjecture for $k=2$. The conjecture is open for general $k$.
\par The partition function $\text{PDO}_t(n)$ has also been studied modulo small powers of $2$. Lin conjectured and Baruah and Kaur \cite{Baruah2020} proved the following congruences modulo $8$: For $n\geq0$, 
\begin{align}\label{Baruah_cong}
\text{PDO}_t(8n+6)\equiv \text{PDO}_t(8n+7)\equiv 0\pmod 8.
\end{align}
Using $2$- and $3$-dissections of certain $q$-products, Vandna and Kaur \cite{V_kaur} also proved several congruences modulo $8$, $16$, $48$, $144$, $288$, $1152$, $6912$. For example, for $n\geq0$,
\begin{align*}
\text{PDO}_t(12n+6)\equiv \text{PDO}_t(48n+38)&\equiv 0\pmod{16},\\
\text{PDO}_t(32n+12)\equiv \text{PDO}_t(96n+76)&\equiv 0\pmod{48},\\
\text{PDO}_t(48n+44)\equiv \text{PDO}_t(96n+60)&\equiv 0\pmod{144}.
\end{align*} 
In \cite[Theorem 1.4]{Barman2023}, jointly with Ajit Singh, using the work of Ono and Taguchi \cite{ono2005} on nilpotency of Hecke operators we proved  that there exist infinitely many congruences modulo arbitrary powers of $2$. To be specific, we proved that there exists an integer $s\geq0$ such that for every $u\geq1$ and distinct primes $p_1,\ldots, p_{s+u}$ coprime to $6$,
\begin{align*}
\text{PDO}_t\left(p_1\cdots p_{s+u}\cdot n \right)\equiv 0\pmod{2^{u}}
\end{align*}
for all nonnegative integer $n$ coprime to $p_1, \ldots, p_{s+u}$. This left the case for the primes $2$ and $3$. In the following theorem, we derive infinitely many congruences for $\text{PDO}_t(n)$ modulo $8$ and $32$, which involve the primes $2$ and $3$. 
\begin{theorem}\label{Thm1}
For any prime $p\equiv-1\pmod{6}$ and $n\geq 0$, we have
\begin{align}\label{Thm1_1}
\emph{PDO}_t\left(6p^2n+6kp+3p^2\right)&\equiv0\pmod{8},\\ \emph{PDO}_t\left(24p^2n+24kp+12p^2\right)&\equiv0\pmod{32},\label{Thm1_2}
\end{align}
where $k=1,2,\ldots, p-1$.
\end{theorem}
We obtain the following corollary to Theorem \ref{Thm1}.
\begin{corollary}\label{Thm1_Cor}
For any prime $p\equiv-1\pmod{6}$, and $n\geq0$, $\ell\geq1$
\begin{align}\label{Thm1_Cor1}
\emph{PDO}_t\left(3^{\ell}(6p^2n+6kp+3p^2)\right)&\equiv0\pmod{8},\\ \emph{PDO}_t\left(3^{\ell}(24p^2n+24kp+12p^2)\right)&\equiv0\pmod{32},\label{Thm1_Cor2}
\end{align}
where $k=1,2,\ldots, p-1$.
\end{corollary}
Also, using Radu's approach \cite{radu1, radu2}, we derive several congruences for $\text{PDO}_t(n)$ modulo different powers of $2$.
\begin{theorem}\label{Thm2}
For all $n\geq0$, we have.
\begin{align} \label{0.00}
\emph{PDO}_t(6n)&\equiv0\pmod{8},\\ \label{0.01}
\emph{PDO}_t(12n)&\equiv0\pmod{16},\\ \label{0.02}
\emph{PDO}_t(24n)&\equiv0\pmod{32},\\ \label{0.03}
\emph{PDO}_t(48n)&\equiv0\pmod{64},\\ \label{0.04}
\emph{PDO}_t(96n)&\equiv0\pmod{128},\\ \label{0.05}
\emph{PDO}_t(192n)&\equiv0\pmod{256},\\ \label{0.06}
\emph{PDO}_t(6n+3)&\equiv0\pmod{4},\\ \label{0.07}
\emph{PDO}_t(12n+6)&\equiv0\pmod{8},\\ \label{0.08}
\emph{PDO}_t(24n+12)&\equiv0\pmod{16},\\ \label{0.09}
\emph{PDO}_t(48n+24)&\equiv0\pmod{32},\\ \label{0.010}
\emph{PDO}_t(96n+48)&\equiv0\pmod{64},\\ \label{0.011}
\emph{PDO}_t(192n+96)&\equiv0\pmod{128},\\ \label{0.012}
\emph{PDO}_t(12n+3)\equiv\emph{PDO}_t(12n+9)&\equiv0\pmod{4},\\ \label{0.013}
\emph{PDO}_t(24n+6)\equiv\emph{PDO}_t(24n+18)&\equiv0\pmod{8},\\ \label{0.014}
\emph{PDO}_t(48n+12)\equiv\emph{PDO}_t(48n+36)&\equiv0\pmod{16},\\ \label{0.015}
\emph{PDO}_t(96n+24)\equiv\emph{PDO}_t(96n+72)&\equiv0\pmod{32},\\ \label{0.016}
\emph{PDO}_t(192n+48)\equiv\emph{PDO}_t(192n+144)&\equiv0\pmod{64}. 
\end{align}
\end{theorem}
We remark that the congruences \eqref{0.00} and \eqref{0.01} can also be derived from \cite[p. 14]{Lin2018}
\begin{align}\label{genfn_6n}
\sum_{n=0}^{\infty}\text{PDO}_t(6n)q^n=16q\frac{f^4_2f^3_3f^4_4}{f^9_1}.
\end{align}
Also, $\text{PDO}_t(24n+6)\equiv0\pmod{8}$ can be seen from \eqref{Baruah_cong}. The congruences \eqref{0.07} and \eqref{0.012} are proved in \cite{V_kaur} using basic $q$-series techniques. Our main motive of collectively writing and proving congruences \eqref{0.00}-\eqref{0.016} is to look for the pattern for $\text{PDO}_t(n)$ modulo powers of $2$ and check the existence of congruences modulo powers of $2$ similar to \eqref{lin_conj1} and \eqref{lin_conj2}. We conjecture some congruences modulo arbitrary powers of $2$ based on the congruences obtained in Theorem \ref{Thm2}.
\begin{conj}\label{conj1}
For $k, n\geq0$,
\begin{align*}
\text{PDO}_t\left(3\cdot2^kn\right)&\equiv0\pmod{2^{k+2}},\\
\text{PDO}_t\left(3\cdot2^{k+1}n+3\cdot2^{k}\right)&\equiv0\pmod{2^{k+2}},\\
\text{PDO}_t\left(3\cdot2^{k+2}n+3\cdot2^{k}\right)&\equiv0\pmod{2^{k+2}},\\
\text{PDO}_t\left(3\cdot2^{k+2}n+3^2\cdot2^{k}\right)&\equiv0\pmod{2^{k+2}}.
\end{align*}
\end{conj}
Next, we present a new approach to prove the congruences in Conjecture \ref{conj_Lin}. We study 
\begin{align}\label{0.11}
\sum_{n=0}^{\infty}\text{PDO}_t\left(8\cdot3^kn\right)q^n
\end{align}
and 
\begin{align}\label{0.12}
\sum_{n=0}^{\infty}\text{PDO}_t\left(12\cdot3^kn\right)q^n
\end{align}
modulo $3^{k+3}$. From \cite[p. 16]{Lin2018}, we have
\begin{align}\label{0.1}
\sum_{n=0}^{\infty}\text{PDO}_t\left(8n\right)q^n\equiv2^2\cdot3^2qf^2_1f^2_2f^2_3f^2_6\pmod{3^3}
\end{align}
and from \cite[eq. (5.12)]{Lin2018}, we have
\begin{align}\label{0.2}
\sum_{n=0}^{\infty}\text{PDO}_t\left(12n\right)q^n\equiv3^2qf^4_6\pmod{3^3}.
\end{align}
From \eqref{0.1} and \eqref{0.2}, we obtain \eqref{lin_conj1} and \eqref{lin_conj2}, respectively, for $k=0$. Using basic $q$-series techniques and $2$- and $3$-dissections of certain $q$-products, we prove the following theorem.
\begin{theorem}\label{Thm3}
We have:
\begin{align}\label{0.3}
\sum_{n=0}^{\infty}\emph{PDO}_t\left(8\cdot3n\right)q^n&\equiv2^3\cdot3^3qf^2_1f^2_2f^2_3f^2_6\pmod{3^4},\\ \sum_{n=0}^{\infty}\emph{PDO}_t\left(12\cdot3n\right)q^n&\equiv2^5\cdot3^3qf^4_6\pmod{3^4}.\label{0.4}
\end{align}
\end{theorem}
With $k=1$, \eqref{lin_conj1} and \eqref{lin_conj2} readily follow from Theorem \ref{Thm3}. For $k\geq2$, the technique we use to prove Theorem \ref{Thm3} becomes tedious to analyze \eqref{0.11} and \eqref{0.12} modulo $3^{k+3}$. Here, we use the theory of modular forms for the case $k=2$. In particular, we have the following theorem.
\begin{theorem}\label{Thm4}
We have:
\begin{align}\label{0.5}
\sum_{n=0}^{\infty}\emph{PDO}_t\left(8\cdot3^2n\right)q^n&\equiv2^4\cdot3^4qf^2_1f^2_2f^2_3f^2_6\pmod{3^5},\\ \sum_{n=0}^{\infty}\emph{PDO}_t\left(12\cdot3^2n\right)q^n&\equiv3^4qf^4_6\pmod{3^5}.\label{0.6}
\end{align}
\end{theorem}
For $k=2$, Conjecture \ref{conj_Lin} follows from Theorem \ref{Thm4}. Observing \eqref{0.1}-\eqref{0.6} and analyzing \eqref{0.11} and \eqref{0.12} for more values of $k$ on \texttt{SageMath}, we propose the following conjecture.
\begin{conj}\label{conj2}
For $k\geq 0$, we have
\begin{align}\label{conj2_1}
\sum_{n=0}^{\infty}\text{PDO}_t\left(8\cdot3^kn\right)q^n&\equiv2^{k+2}\cdot3^{k+2}qf^2_1f^2_2f^2_3f^2_6\pmod{3^{k+3}},\\ \sum_{n=0}^{\infty}\text{PDO}_t\left(12\cdot3^kn\right)q^n&\equiv2^{\alpha_k}\cdot3^{k+2}qf^4_6\pmod{3^{k+3}},\label{conj2_2}
\end{align}
where $\alpha_k:=2k+3$, for odd $k$ and $\alpha_k:=0$, for even $k$.
\end{conj}
In comparison to Conjecture \ref{conj_Lin}, Conjecture \ref{conj2} is more general. Also, a proof of Conjecture \ref{conj2} will immediately settle Conjecture \ref{conj_Lin}. 
\begin{theorem}\label{Thm5}
Conjecture \ref{conj2} implies Conjecture \ref{conj_Lin}.
\end{theorem}
\section{Preliminaries}
We recall some definitions and basic facts on modular forms. For more details, see for example \cite{koblitz1993, ono2004}. We first define the matrix groups 
\begin{align*}
\text{SL}_2(\mathbb{Z}) & :=\left\{\begin{bmatrix}
a  &  b \\
c  &  d      
\end{bmatrix}: a, b, c, d \in \mathbb{Z}, ad-bc=1
\right\},\\
\Gamma_{0}(N) & :=\left\{
\begin{bmatrix}
a  &  b \\
c  &  d      
\end{bmatrix} \in \text{SL}_2(\mathbb{Z}) : c\equiv 0\pmod N \right\},
\end{align*}
\begin{align*}
\Gamma_{1}(N) & :=\left\{
\begin{bmatrix}
a  &  b \\
c  &  d      
\end{bmatrix} \in \Gamma_0(N) : a\equiv d\equiv 1\pmod N \right\},
\end{align*}
and 
\begin{align*}\Gamma(N) & :=\left\{
\begin{bmatrix}
a  &  b \\
c  &  d      
\end{bmatrix} \in \text{SL}_2(\mathbb{Z}) : a\equiv d\equiv 1\pmod N, ~\text{and}~ b\equiv c\equiv 0\pmod N\right\},
\end{align*}
where $N$ is a positive integer. A subgroup $\Gamma$ of $\text{SL}_2(\mathbb{Z})$ is called a congruence subgroup if $\Gamma(N)\subseteq \Gamma$ for some $N$. The smallest $N$ such that $\Gamma(N)\subseteq \Gamma$
is called the level of $\Gamma$. For example, $\Gamma_0(N)$ and $\Gamma_1(N)$
are congruence subgroups of level $N$. 
\par Let $\mathbb{H}:=\{z\in \mathbb{C}: \text{Im}(z)>0\}$ be the upper half of the complex plane. The group $$\text{GL}_2^{+}(\mathbb{R})=\left\{\begin{bmatrix}
a  &  b \\
c  &  d      
\end{bmatrix}: a, b, c, d\in \mathbb{R}~\text{and}~ad-bc>0\right\}$$ acts on $\mathbb{H}$ by $\begin{bmatrix}
a  &  b \\
c  &  d      
\end{bmatrix} z=\displaystyle \frac{az+b}{cz+d}$.  
We identify $\infty$ with $\displaystyle\frac{1}{0}$ and define $\begin{bmatrix}
a  &  b \\
c  &  d      
\end{bmatrix} \displaystyle\frac{r}{s}=\displaystyle \frac{ar+bs}{cr+ds}$, where $\displaystyle\frac{r}{s}\in \mathbb{Q}\cup\{\infty\}$.
This gives an action of $\text{GL}_2^{+}(\mathbb{R})$ on the extended upper half-plane $\mathbb{H}^{\ast}=\mathbb{H}\cup\mathbb{Q}\cup\{\infty\}$. 
Suppose that $\Gamma$ is a congruence subgroup of $\text{SL}_2(\mathbb{Z})$. A cusp of $\Gamma$ is an equivalence class in $\mathbb{P}^1=\mathbb{Q}\cup\{\infty\}$ under the action of $\Gamma$.
\par The group $\text{GL}_2^{+}(\mathbb{R})$ also acts on functions $f: \mathbb{H}\rightarrow \mathbb{C}$. In particular, suppose that $\gamma=\begin{bmatrix}
a  &  b \\
c  &  d      
\end{bmatrix}\in \text{GL}_2^{+}(\mathbb{R})$. If $f(z)$ is a meromorphic function on $\mathbb{H}$ and $\ell$ is an integer, then define the slash operator $|_{\ell}$ by 
$$(f|_{\ell}\gamma)(z):=(\text{det}~{\gamma})^{\ell/2}(cz+d)^{-\ell}f(\gamma z).$$
\begin{definition}
	Let $\Gamma$ be a congruence subgroup of level $N$. A holomorphic function $f: \mathbb{H}\rightarrow \mathbb{C}$ is called a modular form with integer weight $\ell$ on $\Gamma$ if the following hold:
	\begin{enumerate}
		\item We have $$f\left(\displaystyle \frac{az+b}{cz+d}\right)=(cz+d)^{\ell}f(z)$$ for all $z\in \mathbb{H}$ and all $\begin{bmatrix}
		a  &  b \\
		c  &  d      
		\end{bmatrix} \in \Gamma$.
		\item If $\gamma\in \text{SL}_2(\mathbb{Z})$, then $(f|_{\ell}\gamma)(z)$ has a Fourier expansion of the form $$(f|_{\ell}\gamma)(z)=\displaystyle\sum_{n\geq 0}a_{\gamma}(n)q_N^n,$$
		where $q_N:=e^{2\pi iz/N}$.
	\end{enumerate}
In addition, if $a_{\gamma}(0)=0$ for all $\gamma \in \text{SL}_2(\mathbb{Z})$, then $f$ is called a cusp form.
\end{definition}
For a positive integer $\ell$, the complex vector space of modular forms (resp. cusp forms) of weight $\ell$ with respect to a congruence subgroup $\Gamma$ is denoted by $M_{\ell}(\Gamma)$ (resp. $S_{\ell}(\Gamma)$).
\begin{definition}\cite[Definition 1.15]{ono2004}
	If $\chi$ is a Dirichlet character modulo $N$, then we say that a modular form $f\in M_{\ell}(\Gamma_1(N))$ (resp. $S_{\ell}(\Gamma_1(N))$) has Nebentypus character $\chi$ if
	$$f\left( \frac{az+b}{cz+d}\right)=\chi(d)(cz+d)^{\ell}f(z)$$ for all $z\in \mathbb{H}$ and all $\begin{bmatrix}
	a  &  b \\
	c  &  d      
	\end{bmatrix} \in \Gamma_0(N)$. The space of such modular forms (resp. cusp forms) is denoted by $M_{\ell}(\Gamma_0(N), \chi)$ (resp. $S_{\ell}(\Gamma_0(N), \chi)$). 
\end{definition}
In this paper, the relevant modular forms are those that arise from eta-quotients. Recall that the Dedekind eta-function $\eta(z)$ is defined by
\begin{align*}
\eta(z):=q^{1/24}\prod_{n=1}^{\infty}(1-q^n)=q^{1/24}f_1,
\end{align*}
where $q:=e^{2\pi iz}$ and $z\in \mathbb{H}$. A function $f(z)$ is called an eta-quotient if it is of the form
\begin{align*}
f(z)=\prod_{\delta\mid N}\eta(\delta z)^{r_\delta},
\end{align*}
where $N$ is a positive integer and $r_{\delta}$ is an integer. We now recall two theorems from \cite[p. 18]{ono2004} which are very useful in checking modularity of eta-quotients.
\begin{theorem}\cite[Theorem 1.64]{ono2004}\label{thm_ono1} If $f(z)=\prod_{\delta\mid N}\eta(\delta z)^{r_\delta}$ 
	is an eta-quotient such that $\ell=\frac{1}{2}\sum_{\delta\mid N}r_{\delta}\in \mathbb{Z}$, 
	$$\sum_{\delta\mid N} \delta r_{\delta}\equiv 0 \pmod{24}$$ and
	$$\sum_{\delta\mid N} \frac{N}{\delta}r_{\delta}\equiv 0 \pmod{24},$$
	then $f(z)$ satisfies $$f\left( \frac{az+b}{cz+d}\right)=\chi(d)(cz+d)^{\ell}f(z)$$
	for every  $\begin{bmatrix}
	a  &  b \\
	c  &  d      
	\end{bmatrix} \in \Gamma_0(N)$. Here the character $\chi$ is defined by $\chi(d):=\left(\frac{(-1)^{\ell} s}{d}\right)$, where $s:= \prod_{\delta\mid N}\delta^{r_{\delta}}$. 
\end{theorem}
Suppose that $f$ is an eta-quotient satisfying the conditions of Theorem \ref{thm_ono1} and that the associated weight $\ell$ is a positive integer. If $f(z)$ is holomorphic (resp. vanishes) at all of the cusps of $\Gamma_0(N)$, then $f(z)\in M_{\ell}(\Gamma_0(N), \chi)$ (resp. $S_{\ell}(\Gamma_0(N), \chi)$). The following theorem gives the necessary criterion for determining orders of an eta-quotient at cusps.
\begin{theorem}\cite[Theorem 1.65]{ono2004}\label{thm_ono2}
	Let $c, d$ and $N$ be positive integers with $d\mid N$ and $\gcd(c, d)=1$. If $f$ is an eta-quotient satisfying the conditions of Theorem \ref{thm_ono1} for $N$, then the 
	order of vanishing of $f(z)$ at the cusp $\frac{c}{d}$ 
	is $$\frac{N}{24}\sum_{\delta\mid N}\frac{\gcd(d,\delta)^2r_{\delta}}{\gcd(d,\frac{N}{d})d\delta}.$$
\end{theorem}
We now recall a result of Sturm \cite{Sturm1984} which gives a criterion to test whether two modular forms are congruent modulo a given prime.
\begin{theorem}\label{Sturm}
	Let $p$ be a prime number, and $f(z)=\sum_{n=n_0}^\infty a(n)q^n$ and $g(z)=\sum_{n=n_1}^\infty b(n)q^n$ be modular forms of weight $k$ for $\Gamma_0(N)$ of characters $\chi$ and $\psi$, respectively, where $n_0, n_1\geq 0$. If either $\chi=\psi$ and 
	\begin{align*}
		a(n)\equiv b(n)\pmod p~~ \text{for all}~~ n\leq \frac{kN}{12}\prod_{d~~ \text{prime};~~ d|N}\left(1+\frac{1}{d}\right),
	\end{align*}
	or $\chi\neq\psi$ and 
	\begin{align*}
		a(n)\equiv b(n)\pmod p~~ \text{for all}~~ n\leq \frac{kN^2}{12}\prod_{d~~ \text{prime};~~ d|N}\left(1-\frac{1}{d^2}\right),
	\end{align*}
	then $f(z)\equiv g(z)\pmod p$~~ $(i.e.,~~a(n)\equiv b(n)\pmod p~~\text{for all}~~n)$.
\end{theorem}
\par We next recall the definition of $U$-operator, see \cite[p. 28]{ono2004}. For a positive integer $d$, the $U$-operator $U(d)$ is defined by
\begin{align*}
\left(\sum_{n\geq n_0}c(n)q^n\right) \mid U(d):= \sum_{n\geq n_0}c(dn)q^n.
\end{align*}
The following proposition from \cite{ono2004} describes the behavior of $U$-operator.
\begin{prop}{\cite[p. 28]{ono2004}}\label{lemma_U_operator}
Suppose that $f(z)\in M_{\ell}(\Gamma_0(N), \chi)$. If $d\mid N$, then
\begin{align*}
f(z)\mid U(d)\in M_{\ell}(\Gamma_0(N), \chi).
\end{align*}
\end{prop}
\par We finally discuss Radu's technique developed in \cite{radu1, radu2}, that we will use to prove certain congruences for eta-quotients. Define 
\begin{align*}
	\Gamma_{\infty} & :=\left\{
	\begin{bmatrix}
		1  &  n \\
		0  &  1      
	\end{bmatrix}: n\in \mathbb{Z}  \right\}.
\end{align*}
We recall that the index of $\Gamma_{0}(N)$ in $\text{SL}_2(\mathbb{Z})$ is
\begin{align*}
	[\text{SL}_2(\mathbb{Z}) : \Gamma_0(N)] = N\prod_{p|N}(1+p^{-1}),
\end{align*}
where $p$ denotes a prime.
\par 
For a positive integer $M$, let $R(M)$ be the set of integer sequences $r=(r_\delta)_{\delta|M}$ indexed by the positive divisors of $M$. 
If $r \in R(M)$ and $1=\delta_1<\delta_2< \cdots <\delta_k=M$ 
are the positive divisors of $M$, we write $r=(r_{\delta_1}, \ldots, r_{\delta_k})$. Define $c_r(n)$ by 
\begin{align*}
	\sum_{n=0}^{\infty}c_r(n)q^n:=\prod_{\delta|M}(q^{\delta};q^{\delta})^{r_{\delta}}_{\infty}=\prod_{\delta|M}\prod_{n=1}^{\infty}(1-q^{n \delta})^{r_{\delta}}.
\end{align*}
The approach to prove congruences for $c_r(n)$ developed by Radu \cite{radu1, radu2} reduces the number of coefficients that one must check as compared with the classical method which uses Sturm's bound alone.
\par 
Let $m$ be a positive integer. For any integer $s$, let $[s]_m$ denote the residue class of $s$ in $\mathbb{Z}_m:= \mathbb{Z}/ {m\mathbb{Z}}$. 
Let $\mathbb{Z}_m^{*}$ be the set of all invertible elements in $\mathbb{Z}_m$. Let $\mathbb{S}_m\subseteq\mathbb{Z}_m$  be the set of all squares in $\mathbb{Z}_m^{*}$. For $t\in\{0, 1, \ldots, m-1\}$
and $r \in R(M)$, we define a subset $P_{m,r}(t)\subseteq\{0, 1, \ldots, m-1\}$ by
\begin{align*}
	P_{m,r}(t):=\left\{t': \exists [s]_{24m}\in \mathbb{S}_{24m} ~ \text{such} ~ \text{that} ~ t'\equiv ts+\frac{s-1}{24}\sum_{\delta|M}\delta r_\delta \pmod{m} \right\}.
\end{align*}
\begin{definition}
	Suppose $m, M$, and $N$ are positive integers, $r=(r_{\delta})\in R(M)$, and $t\in \{0, 1, \ldots, m-1\}$. Let $k=k(m):=\gcd(m^2-1,24)$ and write  
	\begin{align*}
		\prod_{\delta|M}\delta^{|r_{\delta}|}=2^s\cdot j,
	\end{align*}
	where $s$ and $j$  are nonnegative integers with $j$ odd. The set $\Delta^{*}$ consists of all tuples $(m, M, N, (r_{\delta}), t)$ satisfying these conditions and all of the following.
	\begin{enumerate}
		\item Each prime divisor of $m$ is also a divisor of $N$.
		\item $\delta|M$ implies $\delta|mN$ for every $\delta\geq1$ such that $r_{\delta} \neq 0$.
		\item $kN\sum_{\delta|M}r_{\delta} mN/\delta \equiv 0 \pmod{24}$.
		\item $kN\sum_{\delta|M}r_{\delta} \equiv 0 \pmod{8}$.  
		\item  $\frac{24m}{\gcd{(-24kt-k{\sum_{{\delta}|M}}{\delta r_{\delta}}},24m)}$ divides $N$.
		\item If $2|m$, then either $4|kN$ and $8|sN$ or $2|s$ and $8|(1-j)N$.
	\end{enumerate}
\end{definition}
Let $m, M, N$ be positive integers. For $\gamma=
\begin{bmatrix}
	a  &  b \\
	c  &  d     
\end{bmatrix} \in \text{SL}_2(\mathbb{Z})$, $r\in R(M)$ and $r'\in R(N)$, set 
\begin{align*}
	p_{m,r}(\gamma):=\min_{\lambda\in\{0, 1, \ldots, m-1\}}\frac{1}{24}\sum_{\delta|M}r_{\delta}\frac{\gcd^2(\delta a+ \delta k\lambda c, mc)}{\delta m}
\end{align*}
and 
\begin{align*}
	p_{r'}^{*}(\gamma):=\frac{1}{24}\sum_{\delta|N}r'_{\delta}\frac{\gcd^2(\delta,c)}{\delta}.
\end{align*}
\begin{lemma}\label{lem1}\cite[Lemma 4.5]{radu1} Let $u$ be a positive integer, $(m, M, N, r=(r_{\delta}), t)\in\Delta^{*}$ and $r'=(r'_{\delta})\in R(N)$. 
	Let $\{\gamma_1,\gamma_2, \ldots, \gamma_n\}\subseteq \text{SL}_2(\mathbb{Z})$ be a complete set of representatives of the double cosets of $\Gamma_{0}(N) \backslash \text{SL}_2(\mathbb{Z})/ \Gamma_\infty$. Assume that $p_{m,r}(\gamma_i)+p_{r'}^{*}(\gamma_i) \geq 0$ for all $1 \leq i \leq n$. Let $t_{min}=\min_{t' \in P_{m,r}(t)} t'$ and
	\begin{align*}
		\nu:= \frac{1}{24}\left\{ \left( \sum_{\delta|M}r_{\delta}+\sum_{\delta|N}r'_{\delta}\right)[\text{SL}_2(\mathbb{Z}):\Gamma_{0}(N)] -\sum_{\delta|N} \delta r'_{\delta}\right\}-\frac{1}{24m}\sum_{\delta|M}\delta r_{\delta} 
		- \frac{ t_{min}}{m}.
	\end{align*}	
	If the congruence $c_r(mn+t')\equiv0\pmod u$ holds for all $t' \in P_{m,r}(t)$ and $0\leq n\leq \lfloor\nu\rfloor$, then it holds for all $t'\in P_{m,r}(t)$ and $n\geq0$.
\end{lemma}
To apply Lemma \ref{lem1}, we utilize the following result, which gives us a complete set of representatives of the double coset in $\Gamma_{0}(N) \backslash \text{SL}_2(\mathbb{Z})/ \Gamma_\infty$. 
\begin{lemma}\label{lem2}\cite[Lemma 4.3]{wang} If $N$ or $\frac{1}{2}N$ is a square-free integer, then
	\begin{align*}
		\bigcup_{\delta|N}\Gamma_0(N)\begin{bmatrix}
			1  &  0 \\
			\delta  &  1      
		\end{bmatrix}\Gamma_ {\infty}=\text{SL}_2(\mathbb{Z}).
	\end{align*}
\end{lemma}
The following lemma is an easy consequence of the Binomial theorem.
\begin{lemma}\label{lemma_binom}
For any prime $p$, we have
\begin{align*}
f^p_1&\equiv f_p\pmod{p},\\
f^{p^2}_1&\equiv f^p_p\pmod{p^2}.
\end{align*}
\end{lemma}
We will frequently use the congruences in Lemma \ref{lemma_binom} for $p=2,3$ without explicitly mentioning them. The next lemma contains some known $2$- and $3$-dissections of certain $q$-products.
\begin{lemma}
We have
\begin{align}\label{1.1}
f_1f_3&=\frac{f_2f^2_8f^4_{12}}{f^2_4f_6f^2_{24}}-q\frac{f^4_4f_6f^2_{24}}{f_2f^2_8f^2_{12}},\\
\frac{f_3}{f^3_1}&=\frac{f^6_4f^3_6}{f^9_2f^2_{12}}+3q\frac{f^2_4f_6f^2_{12}}{f^7_2},\label{1.3}\\
f^3_1&=\frac{f_6f^6_9}{f_3f^3_{18}}-3qf^3_9+4q^3\frac{f^2_3f^6_{18}}{f^2_6f^3_9},\label{1.2}\\
\frac{f^2_1}{f_2}&=\frac{f^2_9}{f_{18}}-2q\frac{f_3f^2_{18}}{f_6f_9},\label{1.4}\\
\frac{f_2}{f^2_1}&=\frac{f^4_6f^6_9}{f^8_3f^3_{18}}+2q\frac{f^3_6f^3_9}{f^7_3}+4q^2\frac{f^2_6f^3_{18}}{f^6_3},\label{1.5}\\
\frac{1}{f^3_1}&=\frac{f^3_9}{f^{10}_3}\left(c^2(q^3)+3qc(q^3)\frac{f^3_9}{f_3}+9q^2\frac{f^6_9}{f^2_3}\right), \label{1.6}
\end{align}
where $\displaystyle c(q):=\sum_{m,n\in\mathbb{Z}}q^{m^2+mn+n^2}$.
\end{lemma}
\begin{proof}
Equation \eqref{1.1} is \cite[eq. (30.12.1)]{hirschhorn}. By replacing $q$ with $-q$ in \cite[eq. (22.6.2)]{hirschhorn}, we obtain \eqref{1.3}. Identity \eqref{1.2} is the $3$-dissection formula for triangular numbers \cite[eq. (14.8.5)]{hirschhorn}. Equation \eqref{1.4} is \cite[eq. (14.3.2)]{hirschhorn}. We can deduce \eqref{1.5} from \eqref{1.4} by replacing $q$ with $\omega q$ and $\omega^2q$ and multiplying the two results. Here, $\omega$ is a primitive third root of unity. The last identity \eqref{1.6} is \cite[eq. (39.2.8)]{hirschhorn}.
\end{proof}
\section{Proof of Theorem \ref{Thm1}}
In this section we present the proof of Theorem \ref{Thm1}.
\begin{proof}[Proof of Theorem \ref{Thm1}] We first recall the Jacobi's identity \cite[Theorem 1.3.9]{berndt}
\begin{align}\label{3.0}
f^3_1=\sum_{n=0}^{\infty}(-1)^n(2n+1)q^{n(n+1)/2}.	
\end{align}
From \cite[eq. (5.6)]{Lin2018}, we have
\begin{align}\label{3.1}
\sum_{n=0}^{\infty}\text{PDO}_t(3n)q^n\equiv4qf^3_2f^3_6\pmod{8}.
\end{align}
Extracting the terms with odd powers of $q$ on each side of \eqref{3.1}, and dividing both sides by $q$, and then substituting $q^2$ by $q$, we deduce that
\begin{align}\label{3.2}
\sum_{n=0}^{\infty}\text{PDO}_t(6n+3)q^n\equiv4f^3_1f^3_3\pmod{8}.
\end{align}
Invoking \eqref{3.0} in \eqref{3.2} yields
\begin{align*}
\sum_{n=0}^{\infty}\text{PDO}_t(6n+3)q^n\equiv4\cdot\sum_{m=0}^{\infty}q^{\frac{m(m+1)}{2}}\cdot\sum_{n=0}^{\infty}q^{\frac{3n(n+1)}{2}}\pmod{8},
\end{align*}
which we rewrite as
\begin{align}\label{3.3}
\sum_{n=0}^{\infty}\text{PDO}_t(6n+3)q^n\equiv4\cdot\sum_{m,n\geq0}q^{\frac{m(m+1)}{2}+\frac{3n(n+1)}{2}}\pmod{8}.
\end{align}
For any prime $p\geq3$, the congruence
\begin{align*}
\frac{m(m+1)}{2}+\frac{3n(n+1)}{2}\equiv pk+\frac{p^2-1}{2}\pmod{p^2}
\end{align*}
is equivalent to
\begin{align}\label{eqn-neq-01}
\left(m+\frac{1}{2}\right)^2+3\left(n+\frac{1}{2}\right)^2\equiv 2pk\pmod{p^2}. 
\end{align}
Since $p\geq 3$, we can rewrite \eqref{eqn-neq-01} as
\begin{align}\label{3.5}
x^2+3y^2\equiv 2pk\pmod{p^2},
\end{align}
for some $x,y\in\mathbb{Z}$. From \eqref{3.5}, we obtain
\begin{align}\label{3.6}
x^2+3y^2\equiv 0\pmod{p}.
\end{align}
We notice from \eqref{3.6} that $p\mid x$ if and only if $p\mid y$. Now, if $p\mid x$ and $p\mid y$, it follows from \eqref{3.5} that $p\mid k$. If $p\nmid x$ and $p\nmid y$, then from \eqref{3.6} we find that 
\begin{align*}
\frac{x^2}{y^2}\equiv-3\pmod{p},
\end{align*}
which is true if $-3$ is a quadratic residue modulo $p$. Therefore, $x$ and $y$ in \eqref{3.5} exist only if either $p\mid k$ or $-3$ is a quadratic residue modulo $p$. In other words, no such $x$ and $y$ exist if $-3$ is a quadratic nonresidue modulo $p$ and $1\leq k\leq p-1$. Also, $-3$ is a quadratic nonresidue modulo $p$ precisely for the primes $p\equiv-1\pmod{6}$. Hence, employing above analysis in \eqref{3.3}, we conclude that for all $n\geq0$,
\begin{align*}
\text{PDO}_t\left(6\left(p^2n+pk+\frac{p^2-1}{2}\right)+3\right)\equiv0\pmod{8}, 
\end{align*}
for any prime $p\equiv-1\pmod{6}$ and $1\leq k\leq p-1$. This completes the proof of \eqref{Thm1_1}.
\par Next, we prove \eqref{Thm1_2}. 
From \eqref{genfn_6n}, we have
\begin{align}\label{3.9}
\sum_{n=0}^{\infty}\text{PDO}_t(6n)q^n\equiv16qf^7_2f_6(f_1f_3)\pmod{32}.
\end{align}
Substituting \eqref{1.1} in \eqref{3.9}, we obtain
\begin{align}\label{3.10}
\sum_{n=0}^{\infty}\text{PDO}_t(6n)q^n\equiv16\left(qf^{12}_2+q^2f^6_2f^6_6\right)\pmod{32}.
\end{align}
Extracting the terms with even powers of $q$ on each side of \eqref{3.10}, and substituting $q^2$ by $q$ yields
\begin{align}\label{3.11}
\sum_{n=0}^{\infty}\text{PDO}_t(12n)q^n\equiv16qf^3_2f^3_6\pmod{32}.
\end{align}
Extracting the terms with odd powers of $q$ on each side of \eqref{3.11}, and dividing both sides by $q$ and then substituting $q^2$ by $q$, we obtain
\begin{align*}
\sum_{n=0}^{\infty}\text{PDO}_t(24n+12)q^n\equiv16qf^3_1f^3_3\pmod{32}.
\end{align*}
From here, the proof of \eqref{Thm1_2} goes on the similar steps as shown in the proof of \eqref{Thm1_1}. This completes the proof of Theorem \ref{Thm1}.
\end{proof}
\begin{proof}[Proof of Corollary \ref{Thm1_Cor}]
Substituting the value of $f^3_2$ obtained from \eqref{1.2} in \eqref{3.1}, we obtain
\begin{align}\label{3.14}
\sum_{n=0}^{\infty}\text{PDO}_t(3n)q^n\equiv4qf^4_6-12q^3f^3_6f^3_{18}\pmod{8}.
\end{align}
Extracting those terms of the form $q^{3n}$ on both sides of \eqref{3.14} and replacing $q^3$ by $q$, we arrive at
\begin{align}\label{3.15}
\sum_{n=0}^{\infty}\text{PDO}_t(9n)q^n\equiv4qf^3_2f^3_6\pmod{8}.
\end{align}
From \eqref{3.1} and \eqref{3.15}, we conclude that for all $n\geq0$,	
\begin{align*}
\text{PDO}_t(3n)\equiv\text{PDO}_t(9n)\pmod{8},
\end{align*}
which by induction implies that for all $\ell\geq1$ and $n\geq0$,
\begin{align}\label{3.16}
\text{PDO}_t(3n)\equiv\text{PDO}_t(3^{\ell}n)\pmod{8}.
\end{align}
Then, \eqref{Thm1_Cor1} follows from \eqref{Thm1_1} and \eqref{3.16}.
\par Next, we prove \eqref{Thm1_Cor2}. Substituting the value of $f^3_2$ obtained from \eqref{1.2} in \eqref{3.11}, and then extracting those terms of the form $q^{3n}$ on both sides, we obtain
\begin{align}\label{3.17}
\sum_{n=0}^{\infty}\text{PDO}_t(36n)q^{3n}\equiv16q^3f^3_6f^3_{18}\pmod{32}.
\end{align}
Replacing $q^3$ by $q$ in \eqref{3.17} yields
\begin{align}\label{3.18}
\sum_{n=0}^{\infty}\text{PDO}_t(36n)q^{n}\equiv16qf^3_2f^3_6\pmod{32}.
\end{align}
From \eqref{3.11} and \eqref{3.18}, we conclude that for all $n\geq0$,	
\begin{align*}
\text{PDO}_t(12n)\equiv\text{PDO}_t(36n)\pmod{32},
\end{align*}
which by induction implies that for all $\ell\geq1$ and $n\geq0$,
\begin{align}\label{3.19}
\text{PDO}_t(12n)\equiv\text{PDO}_t(3^{\ell}\cdot12n)\pmod{32}.
\end{align}
Then, \eqref{Thm1_Cor2} follows from \eqref{Thm1_2} and \eqref{3.19}. This completes the proof of the corollary.
\end{proof}
\section{Proof of Theorem \ref{Thm2}}
In this section we prove Theorem \ref{Thm2}. 
\begin{proof}[Proof of Theorem \ref{Thm2}]
Define
\begin{align}\label{4.1}
\sum_{n=0}^{\infty}a(n)q^n:=\frac{f_2f^2_3f^2_{12}}{f^2_1f_6}.
\end{align}
From \eqref{gen_fn_PDOt} and \eqref{4.1}, we find that, for all $n\geq0$,
\begin{align}\label{4.2}
\text{PDO}_t(n+1)=a(n).
\end{align}
For all the values of $m$ and $t$ listed in Table \ref{Table} and $M=12$, $N=12$, and $r=(-2,1,2,0,-1,2)$, we verify that $(m,M,N,r,t) \in \Delta^{*}$. By Lemma \ref{lem2}, we know that the set $\left\{\begin{bmatrix}
	1  &  0 \\
	\delta  &  1      
\end{bmatrix}:\delta|12 \right\}$ forms a complete set of double coset representatives of $\Gamma_{0}(12) \backslash \text{SL}_2(\mathbb{Z})/ \Gamma_\infty$. Let $\gamma_{\delta}=\begin{bmatrix}
1  &  0 \\
\delta  &  1      
\end{bmatrix}$. For every triple $(m,t,r')$ from Table \ref{Table}, we use \texttt{SageMath} to verify that $p_{m,r}(\gamma_{\delta})+p_{r'}^{*}(\gamma_{\delta}) \geq 0$ for each $\delta\mid12$. We find that for each case the set $P_{m,r}(t)=\left\lbrace t\right\rbrace $. We then compute the upper bound $\lfloor\nu\rfloor$ in each case, as in Lemma \ref{lem1}. Using \texttt{SageMath}, we verify that $a(mn+t)\equiv0\pmod{u}$ for all $m,t,u$ in Table \ref{Table} and for $n\leq \lfloor\nu\rfloor$. By Lemma \ref{lem1}, we derive that $a(mn+t)\equiv0\pmod{u}$ for all $n\geq 0$, and for all $m,t,u$ mentioned in Table \ref{Table}. Then, from \eqref{4.2}, we conclude that for all $n\geq0$ and $m,t,u$ in Table \ref{Table},
\begin{align*}
\text{PDO}_t(mn+t+1)\equiv0\pmod{u}.
\end{align*}
This proves all the congruences \eqref{0.00}-\eqref{0.016}.
\end{proof}
\begin{center}
\begin{table}[hbt!]
\caption{}
\begin{tabular}{|p{.4cm}|p{0.4cm}|p{2.2cm}|p{0.5cm}|p{0.5cm}|p{.4cm}|p{0.5cm}|p{2.3cm}|p{0.5cm}|p{0.5cm}|}
\hline
$m$ & $t$ & $r'$ & $\lfloor\nu\rfloor$ & $u$ & $m$ & $t$ & $r'$ & $\lfloor\nu\rfloor$ & $u$\\
\hline
$6$ & $2$ & $(5,0,0,0,0,0)$ & $6$ & $4$ & $48$ & $11$ & $(40,0,0,0,0,0)$ & $40$ & $16$ \\
 & $5$ &  & $6$ & $8$& & $23$ &  & $39$ & $32$\\
 & & & & & & $35$ &  & $39$ & $16$\\
 & & & & & & $47$ &  & $39$ & $64$ \\
\hline
$12$ & $2$ & $(10,0,0,0,0,0)$ & $11$ & $4$ & $96$ & $23$ & $(80,0,0,0,0,0)$ & $78$ & $32$\\
& $5$ &  & $11$ & $8$ & & $47$ &  & $78$ & $64$\\
& $8$ &  & $10$ & $4$ & & $71$ &  & $77$ & $32$\\
& $11$ &  & $10$ & $16$ & & $95$ &  & $77$ & $128$\\	
\hline	
$24$ & $5$ & $(20,0,0,0,0,0)$ & $20$ & $8$&$192$ & $47$ & $(160,0,0,0,0,0)$ & $155$ & $64$\\
& $11$ &  & $20$ & $16$& & $95$ &  & $154$ & $128$\\
& $17$ &  & $20$ & $8$& & $143$ &  & $154$ & $64$\\
& $23$ &  & $20$ & $32$& & $191$ &  & $154$ & $256$\\	
\hline	
\end{tabular}
\label{Table}
\end{table}
\end{center}
\section{$\text{PDO}_t(n)$ modulo powers of $3$}
In this section we prove our results on $\text{PDO}_t(n)$ modulo powers of $3$. We first prove Theorem \ref{Thm3}. 
\begin{proof}[Proof of Theorem \ref{Thm3}] 
From \cite[p. 16]{Lin2018}, we have
\begin{align}\label{5.2}
\sum_{n=0}^{\infty}\text{PDO}_t(8n)q^n&\equiv2^2\cdot3^2q\frac{f^8_2f^7_3}{f^{13}_1}\nonumber \\ 
&\equiv2^2\cdot3^2q\frac{f^9_2f^7_3}{f^{18}_1}\left(\frac{f^5_1}{f_2}\right) \nonumber \\
&\equiv2^2\cdot3^2qf^3_6f_3\left(f^3_1\right) \left(\frac{f^2_1}{f_2}\right)\pmod{3^4}.
\end{align}
Invoking \eqref{1.2} and \eqref{1.4} in \eqref{5.2}, and then extracting those terms of the form $q^{3n}$ from both sides, we obtain
\begin{align}\label{5.3}
\sum_{n=0}^{\infty}\text{PDO}_t(8\cdot3n)q^{3n}\equiv2^3\cdot3^3q^3f^2_3f^2_6f^2_9f^2_{18}\pmod{3^4}.
\end{align}
Replacing $q^3$ by $q$ in \eqref{5.3} yields \eqref{0.3}.
\par From \eqref{genfn_6n}, we have
\begin{align}\label{5.5}
\sum_{n=0}^{\infty}\text{PDO}_t(6n)q^{n}=2^4qf^4_2f^4_4\left(\frac{f_3}{f^3_1}\right)^3.
\end{align}
Substituting \eqref{1.3} in \eqref{5.5}, and extracting those terms of the form $q^{2n}$ from both sides, we arrive at
\begin{align}\label{5.6}
\sum_{n=0}^{\infty}\text{PDO}_t(12n)q^{2n}=2^4\cdot3^3q^4\frac{f^{10}_4f^3_6f^6_{12}}{f^{17}_2}+2^4\cdot3^2q^2\frac{f^{18}_4f^7_6}{f^{21}_2f^2_{12}}.
\end{align}
Substituting $q^2$ by $q$ in \eqref{5.6} and taking modulo $3^4$, we obtain
\begin{align}\label{5.7}
\sum_{n=0}^{\infty}\text{PDO}_t(12n)q^{n}&\equiv2^4\cdot3^3q^2\frac{f^{9}_2f^3_3f^6_{6}}{f^{15}_1}\left(\frac{f_2}{f^2_1}\right) +2^4\cdot3^2q\frac{f^{18}_2f^7_3}{f^{18}_1f^2_{6}}\left( \frac{1}{f^3_1}\right)\nonumber\\
&\equiv 2^4\cdot3^3q^2\frac{f^9_{6}}{f^{2}_3}\left(\frac{f_2}{f^2_1}\right) +2^4\cdot3^2qf_3f^4_6\left( \frac{1}{f^3_1}\right)\pmod{3^4}.
\end{align}
Using \eqref{1.5} and \eqref{1.6} in \eqref{5.7}, and extracting terms of the form $q^{3n}$, we arrive at
\begin{align}\label{5.8}
\sum_{n=0}^{\infty}\text{PDO}_t(12\cdot3n)q^{3n}\equiv2^5\cdot3^3q^3f^{12}_6+      2^4\cdot3^4q^3\frac{f^4_6f^9_9}{f^{11}_3}\pmod{3^4}.
\end{align}
Finally, substituting $q^3$ by $q$ in \eqref{5.8} yields \eqref{0.4}.
\end{proof}
\begin{proof}[Proof of Theorem \ref{Thm4}]
From \cite[p. 16]{Lin2018}, we have
\begin{align*}
F_1(z):=\sum_{n=0}^{\infty}\text{PDO}_t(8n)q^n=36q\frac{f^8_2f^7_3}{f^{13}_1}.
\end{align*}
For $k\geq0$, we define
\begin{align*}
A_{k,1}(z):=36\frac{\eta^{3^{k+3}-13}(z)\eta^8(2z)}{\eta^{3^{k+2}-7}(3z)}
\end{align*}
and
\begin{align*}
B_{k,1}(z):=2^{k+2}\cdot3^{k+2}\frac{\eta^{3^{k+3}-13}(z)\eta^8(2z)}{\eta^{3^{k+2}-7}(3z)}.
\end{align*}
For $k\geq0$, we have
\begin{align*}
A_{k,1}(z)\equiv F_1(z)\pmod{3^{k+3}}
\end{align*}
and
\begin{align}\label{6.0a}
B_{k,1}(z)\equiv 2^{k+2}\cdot3^{k+2}qf^2_1f^2_2f^2_3f^2_6\pmod{3^{k+3}}.
\end{align}
We prove that both $A_{k,1}(z)$ and $B_{k,1}(z)$ are modular forms of level $18$ and weight $3^{k+2}+1$ with trivial character $\chi_0$ modulo $18$. By Theorem \ref{thm_ono1}, $A_{k,1}(z)$ and $B_{k,1}(z)$ are eta-quotients of level $N=18$. The cusps of $\Gamma_{0}(18)$ are represented by $\frac{c}{d}$, where $d\mid18$ and $\gcd(c,d)=1$, see for example \cite[p. 5]{ono1996}. By Theorem \ref{thm_ono2}, $A_{k,1}(z)$ is holomorphic at the cusp $\frac{c}{d}$ if and only if
\begin{align*}
(3^{k+3}-13)+\frac{\gcd(d,2)^2}{2}(8)-\frac{\gcd(d,3)^2}{3}(3^{k+2}-7)\geq 0,
\end{align*}
which is true since 
\begin{align*}
(3^{k+3}-13)-\frac{\gcd(d,3)^2}{3}(3^{k+2}-7)\geq0,
\end{align*}
for every $d\mid18$. Similarly, $B_{k,1}(z)$ is holomorphic at every cusp. Thus, by Theorems \ref{thm_ono1} and \ref{thm_ono2}, we find that $A_{k,1}(z),B_{k,1}(z)\in M_{3^{k+2}+1}\left(\Gamma_{0}(18), \chi_{0}\right)$. 
\par Now, we define
\begin{align*}
F_{k,1}(z):=A_{k,1}(z)\mid U^{k}(3),
\end{align*}
where $U^k$ means applying $U$-operator $k$ times. Since $3$ divides the level of the modular form $A_{k,1}(z)$, by Proposition \ref{lemma_U_operator}, $F_{k,1}(z)\in M_{3^{k+2}+1}\left(\Gamma_{0}(18), \chi_{0}\right)$. By Theorem \ref{Sturm}, the Sturm bound for $M_{3^{k+2}+1}\left(\Gamma_{0}(18), \chi_{0}\right)$ is $3(3^{k+2}+1)$. 
\par We have
\begin{align}\label{6.0}
F_{k,1}(z)&=A_{k,1}(z)\mid U^{k}(3)\nonumber \\
          &\equiv F_1(z)\mid U^{k}(3)\nonumber 
          \\
          &\equiv\sum_{n=0}^{\infty}\text{PDO}_t(8\cdot3^kn)q^n\pmod{3^{k+3}}.
\end{align}
From \eqref{6.0a} and \eqref{6.0}, we find that
\begin{align}\label{eqn-new-02}
F_{k,1}(z)\equiv B_{k,1}(z)\pmod{3^{k+3}}
\end{align}
if and only if
\begin{align}\label{6.1}
\sum_{n=0}^{\infty}\text{PDO}_t(8\cdot3^kn)q^n&\equiv2^{k+2}\cdot3^{k+2}qf^2_1f^2_2f^2_3f^2_6\pmod{3^{k+3}}.
\end{align}
Therefore, to prove the congruence \eqref{0.5}, it is enough to establish \eqref{eqn-new-02} for $k=2$. In view of Theorem \ref{Sturm}, we verify \eqref{eqn-new-02} for $k=2$ up to the Sturm bound using \texttt{SageMath}, and this completes the proof of \eqref{0.5}.	
\par Next, we prove \eqref{0.6}. From \cite[eq. (5.4)]{Lin2018}, we have
\begin{align*}
F_2(z):=\sum_{n=0}^{\infty}\text{PDO}_t(4n)q^n=6q\frac{f^3_2f^2_3f^3_6}{f^6_1}.
\end{align*}
For $k\geq0$, we define
\begin{align*}
A_{k,2}(z):=6\frac{\eta^{3^{k+2}-6}(z)\eta^3(2z)\eta^3(6z)}{\eta^{3^{k+1}-2}(3z)}
\end{align*}
and
\begin{align*}
B_{k,2}(z):=2^{\beta_k}\cdot3^{k+1}\frac{\eta^{3^{k+2}-6}(z)\eta^3(2z)\eta^3(6z)}{\eta^{3^{k+1}-2}(3z)},
\end{align*}
where $\beta_k:=2k+1$, for even $k$ and $\beta_k:=0$, for odd $k$. For all $k\geq0$, we have
\begin{align*}
A_{k,2}(z)\equiv F_2(z)\pmod{3^{k+2}},
\end{align*}
and
\begin{align}\label{6.1a}
B_{k,2}(z)\equiv2^{\beta_k}\cdot3^{k+1}qf^4_6\pmod{3^{k+2}}.
\end{align}
We prove that both $A_{k,2}(z)$ and $B_{k,2}(z)$ are modular forms of level $36$ and weight $3^{k+1}+1$ with trivial character $\chi_0$ modulo $36$. By Theorem \ref{thm_ono1}, $A_{k,2}(z)$ and $B_{k,2}(z)$ are eta-quotients of level $N=36$. The cusps of $\Gamma_{0}(36)$ are represented by $\frac{c}{d}$, where $d\mid36$ and $\gcd(c,d)=1$. By Theorem \ref{thm_ono2}, $A_{k,2}(z)$ is holomorphic at the cusp $\frac{c}{d}$ if and only if
\begin{align*}
(3^{k+2}-6)+\frac{\gcd(d,2)^2}{2}(3)-\frac{\gcd(d,3)^2}{3}(3^{k+1}-2)+\frac{\gcd(d,6)^2}{6}(3)\geq 0,
\end{align*}
which is true since 
\begin{align*}
(3^{k+2}-6)-\frac{\gcd(d,3)^2}{3}(3^{k+1}-2)\geq0,
\end{align*}
for every $d\mid36$. Similarly, $B_{k,2}(z)$ is holomorphic at every cusp. Thus, by Theorems \ref{thm_ono1} and \ref{thm_ono2}, we find that $A_{k,2}(z),B_{k,2}(z)\in M_{3^{k+1}+1}\left(\Gamma_{0}(36), \chi_{0}\right)$. 
\par Now, we define
\begin{align*}
F_{k,2}(z):=A_{k,2}(z)\mid U^{k}(3).
\end{align*}
Since $3$ divides the level of the modular form $A_{k,2}(z)$, by Proposition \ref{lemma_U_operator}, $F_{k,2}(z)\in M_{3^{k+1}+1}\left(\Gamma_{0}(36), \chi_{0}\right)$. By Theorem \ref{Sturm}, the Sturm bound for $M_{3^{k+1}+1}\left(\Gamma_{0}(36), \chi_{0}\right)$ is $6(3^{k+1}+1)$. 
\par We have
\begin{align}\label{6.1b}
F_{k,2}(z)&=A_{k,2}(z)\mid U^{k}(3)\nonumber\\
&\equiv F_2(z)\mid U^{k}(3)\nonumber\\
&\equiv \sum_{n=0}^{\infty}\text{PDO}_t(4\cdot3^kn)q^n\pmod{3^{k+2}}.
\end{align}
From \eqref{6.1a} and \eqref{6.1b}, we find that 
\begin{align}\label{6.1c}
F_{k,2}(z)\equiv B_{k,2}(z)\pmod{3^{k+2}}
\end{align}
if and only if
\begin{align}\label{6.1d}
\sum_{n=0}^{\infty}\text{PDO}_t(4\cdot3^kn)q^n&\equiv2^{\beta_k}\cdot3^{k+1}qf^4_6\pmod{3^{k+2}}.
\end{align}
Therefore, to prove the congruence \eqref{0.6}, it is enough to establish \eqref{6.1c} for $k=3$. In view of Theorem \ref{Sturm}, we verify \eqref{6.1c} for $k=3$ up to the Sturm bound using \texttt{SageMath}, and this completes the proof of \eqref{0.6}.
\end{proof}
\begin{remark}
We can establish \eqref{0.1}-\eqref{0.4} by using smaller values of $k$ in the above proof. Also, for higher values of $k$ verification of \eqref{6.1} and \eqref{6.1d} becomes difficult because of large Sturm bound. We verify \eqref{6.1} and \eqref{6.1d} for few more values of $k$ up to the computable amount and propose Conjecture \ref{conj2}.
\end{remark}
\begin{proof}[Proof of Theorem \ref{Thm5}]
We prove the theorem by induction on $k$. Suppose \eqref{conj2_1} is true for all $k\geq0$. Let \eqref{lin_conj1} hold for some $k$. Then, from \eqref{conj2_1}, we have
\begin{align}\label{7.1}
\sum_{n=0}^{\infty}\text{PDO}_t\left(8\cdot3^{k+1}n\right)q^n&\equiv2^{k+3}\cdot3^{k+3}qf^2_1f^2_2f^2_3f^2_6\nonumber \\
&\equiv2\cdot3\sum_{n=0}^{\infty}\text{PDO}_t\left(8\cdot3^{k}n\right)q^n\pmod{3^{k+4}}.
\end{align}
By our induction hypothesis, the right hand side of \eqref{7.1} is divisible by $3^{k+3}$. Therefore, 
\begin{align*}
\text{PDO}_t(8\cdot3^{k+1}n)\equiv0\pmod{3^{k+3}}.
\end{align*}
Hence, \eqref{lin_conj1} is true for $k+1$. If \eqref{conj2_2} is true, \eqref{lin_conj2} follows in a similar way. This completes the proof.
\end{proof}
\section{Concluding remark}
By Conjecture \ref{conj2}, we also understand the co-existence of congruences \eqref{lin_conj1} and \eqref{lin_conj2}, since from \eqref{conj2_1} and \eqref{conj2_2}, we have
\begin{align}\label{0.1_coexist}
2^{\alpha_k}f^4_2\cdot\sum_{n=0}^{\infty}\text{PDO}_t\left(8\cdot3^kn\right)q^n\equiv2^{k+2}f^8_1\cdot\sum_{n=0}^{\infty}\text{PDO}_t\left(12\cdot3^kn\right)q^n\pmod{3^{k+2}}.
\end{align} 
It is clear from \eqref{0.1_coexist} that if Conjecture \ref{conj2} is true then \eqref{lin_conj1} holds if and only if \eqref{lin_conj2} holds.
\section{Acknowledgements}
The second author gratefully acknowledges the Department of Science and Technology, government of India, for the Core Research Grant (CRG/2021/00314) of SERB.


\begin{thebibliography}{999}
\bibitem{andrews2002}
G. E. Andrews, R. P. Lewis, and J. Lovejoy, {\it Partitions with designated summands}, Acta Arith. 105 (2002), 51--66.

\bibitem{Barman2023}
R. Barman, G. Singh, and A. Singh, {\it Divisibility of the partition function $\text{PDO}_t(n)$ by powers of $2$ and $3$}, Bull. Aust. Math. Soc (2023), 1--12. doi:10.1017/S0004972723000199

\bibitem{Baruah2020}
N. D. Baruah and M. Kaur, {\it New congruences modulo 2, 4, and 8 for the number of tagged parts over the partitions with designated summands}, Ramanujan J. 52 (2020), 253--274.

\bibitem{Baruah2015}
N. D. Baruah and K. K. Ojah, {\it Partitions with designated summands in which all parts are odd}, Integers 15 (2015), Article no. A9.

\bibitem{berndt}
B. C. Berndt, {\it Number Theory in the Spirit of Ramanujan}, Amer. Math. Soc., Providence, RI, 2006.

\bibitem{chen_ji_2013}
W. Y. C. Chen, K. Q. Ji, H. -T. Jin, and E. Y. Y. Shen, {\it On the number of partitions with designated
summands}, J. Number Theory 133 (2013), 2929--2938.

\bibitem{hirschhorn}
M. D. Hirschhorn, {\it The Power of $q$}, Springer, Berlin, 2017.

\bibitem{koblitz1993}
N. Koblitz, {\it Introduction to elliptic curves and modular forms}, Graduate Texts in Mathematics, vol. 97, Springer-Verlag, New York, 1984.

\bibitem{Lin2018}
B. L. S. Lin, {\it The number of tagged parts over the partitions with designated summands}, J. Number Theory 184 (2018), 216--234.

\bibitem{ono1996} 
K. Ono, {\it Parity of the partition function in arithmetic progressions}, J. Reine Angew. Math. 472 (1996), 1--15.

\bibitem{ono2004}
K. Ono, {\it The web of modularity: arithmetic of the coefficients of modular forms and $q$-series}, CBMS Regional Conference Series in Mathematics, vol. 102, Amer. Math. Soc., Providence, RI, 2004.

\bibitem{ono2005} 
K. Ono and Y. Taguchi, {\it 2-adic properties of certain modular forms and their applications to arithmetic functions}, Int. J. Number Theory 1 (2005), 75--101.

\bibitem{radu1}
S. Radu, {\it An algorithmic approach to Ramanujan's congruences}, Ramanujan J.  20 (2)  (2009), 295--302.

\bibitem{radu2}
S. Radu and J. A. Sellers, {\it Congruence properties modulo 5 and 7 for the pod  function}, Int. J. Number Theory 7 (8) (2011), 2249--2259.

\bibitem{Sturm1984} 
J. Sturm, {\it On the congruence of modular forms}, Springer Lect. Notes Math. 1240 (1984), 275--280.

\bibitem{V_kaur}
Vandna and M. Kaur, {\it Some new results on the number of tagged parts over the partitions with designated summands in which all parts are odd}, J. Ramanujan Math. Soc. 37 (2022), 419--435.
\bibitem{wang}
L. Wang, {\it Arithmetic properties of $(k, \ell)$-regular bipartitions}, Bull. Aust. Math. Soc.  95 (2017), 353--364.

\bibitem{xia_2016}
E. X. W. Xia, {\it Arithmetic properties of partitions with designated summands}, J. Number Theory 159 (2016), 160--175. 

\end{thebibliography}
\end{document}